\documentclass[10pt]{amsart}
\usepackage[margin=1.4in]{geometry}
\usepackage{accents}

\usepackage{amsthm, amscd}
\usepackage{amssymb}
\usepackage[nointegrals]{wasysym}
\usepackage[all]{xy}
\usepackage{hyperref}

\numberwithin{equation}{section}
\numberwithin{figure}{section}
\theoremstyle{plain}

\newtheorem{thm}{Theorem}[section]
\newtheorem{lem}[thm]{Lemma}
\newtheorem{conj}[thm]{Conjecture}
\newtheorem{defn}[thm]{Definition}
\newtheorem{prop}[thm]{Proposition}
\newtheorem{cor}[thm]{Corollary}
\newtheorem{rmk}[thm]{Remark}

\renewcommand{\AA}{\mathcal{A}}
\newcommand{\BB}{\mathcal{B}}
\newcommand{\cc}{\mathcal{C}}
\newcommand{\dd}{\mathcal{D}}

\newcommand{\FF}{\mathcal{F}}
\newcommand{\HH}{\mathcal{H}}

\newcommand{\KK}{\mathbb{K}}

\newcommand{\pp}{\mathbb{P}}
\newcommand{\QQ}{\mathcal{Q}}

\newcommand{\WW}{\mathcal{W}}

\newcommand{\MM}{\mathcal{M}}
\newcommand{\NN}{\mathcal{N}}
\newcommand{\CC}{\mathbb{C}}
\newcommand{\RR}{\mathbb{R}}
\newcommand{\TT}{\mathcal{T}}
\newcommand{\VV}{\mathcal{V}}
\newcommand{\YY}{\mathcal{Y}}
\newcommand{\ZZ}{\mathbb{Z}}
\newcommand{\zz}{\mathcal{Z}}

\renewcommand{\o}{\omega}
\newcommand{\mo}{(M,\o)}
\newcommand{\ainf}{A_{\infty}}

\newcommand{\id}{\operatorname{id}}
\newcommand{\End}{\operatorname{End}}

\begin{document}

\title{Projective twists in $A_{\infty}$-categories}
\author{Richard M. Harris}
\address{Centre for Mathematical Sciences, DPMMS, Wilberforce Road, Cambridge,
CB3 0WB.}
\email{r.harris@dpmms.cam.ac.uk}

\begin{abstract}
Given a Lagrangian $V \cong \CC\pp^n$ in a symplectic manifold $\mo$, there is an associated symplectomorphism $\phi_V$ of $M$.  We define the notion of a $\CC\pp^n$-object in an $\ainf$-category $\AA$ and use this to construct algebraically an $\ainf$-functor $\Phi_V$ and prove that it induces an autoequivalence of the derived category $D\AA$.  We conjecture that $\Phi_V$ corresponds to the action of $\phi_V$ and prove this in the lowest dimension $n=1$.  The construction is designed to be mirror to a construction of Huybrechts and Thomas.
\end{abstract}

\maketitle

\section{Introduction}
\label{sec:intro}
Although this paper is a piece of pure algebra, it is motivated by a particular
construction in symplectic topology.  Of central importance in modern
symplectic topology is the idea that we should associate to a symplectic
manifold $(M,\o)$ an $\ainf$-category called its \emph{Fukaya category} $\FF(M,\o)$.
Very briefly, $\FF(M,\o)$ should have as objects Lagrangian submanifolds $L_i$ (with some extra data), the
hom-spaces $\hom_{\FF(M,\o)}(L_0,L_1)$ should be Floer cochain groups, generated by the intersections of
$L_0$ and $L_1$, and the $\ainf$-maps should count holomorphic polygons with
certain Lagrangian boundary conditions.  In full generality, this construction
cannot always be completely carried though \cite{fooo1}, but can be made fully rigorous in certain cases, see for example
\cite{seidelbible}.

One of the fundamental principles behind this construction is that automorphisms of $\mo$ should induce auto-equivalences of $\FF \mo$; more specifically there should be a canonical map
\begin{equation}
\label{eqn:symptofuk}
Aut^c \mo/Ham \mo \to Auteq(D\FF \mo)/\langle [1] \rangle,
\end{equation}
where on the left $Aut^c$ is some subgroup of the full automorphism group that preserves the extra structure needed to define the Fukaya category (such as the homotopy class of the trivialization of the bicanonical bundle $K_M^{\otimes2}$) and we quotient out by the group of Hamiltonian symplectomorphisms, and on the right we quotient out by the shift autoequivalence.  Here $D\FF \mo$ is the so-called \emph{derived Fukaya category}, a triangulated category which is obtained from $\FF \mo$ by a purely algebraic process.  This construction as well as the other relevant algebraic background material will be covered in Section \ref{sec:ainf}.

The best demonstration to date of this principle comes from Dehn twists: given a Lagrangian sphere $V \subset M$ with a choice of diffeomorphism $f\colon V \to S^n$, there is a symplectomorphism $\tau_V$ called the \emph{Dehn twist} about $V$ \cite{seidelles} (the definition of $\tau_V$ requires certain choices, but the result is well-defined in $Aut^c \mo /Ham \mo$).  Algebraically, we can also define the notion of a spherical object $V$ in an $\ainf$-category $\AA$ and define a related functor $T_V \colon D\AA \to D\AA$.  Seidel \cite{seidelbible,seidelles} has proven that, given another Lagrangian $L$, $\tau_V L$ and $T_V L$ give rise to isomorphic objects in $D\FF \mo$ (here we are slightly abusing notation by letting $L$ denote both a submanifold of $M$ and an object of $\FF(M,\o)$). 

The existence of Dehn twists relies on the fact that the geodesic flow on the
round sphere is periodic, and there is a related construction that defines ``twist'' maps for any Lagrangian submanifold admitting such a metric \cite{seidelgradings}.  In particular, this
paper will be concerned with the projective twist $\phi_V$ associated to a Lagrangian $ V \cong \CC
\pp^n$.  

In Section \ref{sec:twists}, we shall define the notion of a $\CC\pp^n$-object in an $\ainf$-category and in the case where $\AA$ is a triangulated $\ainf$-category (see Section \ref{sec:ainf}) we use $V$ to define a functor $\Phi_V \colon \AA \to \AA$.  In Section \ref{sec:equivalence} we prove that $\Phi_V$ induces an auto-equivalence of $D\AA$.  This result is the first step towards proving the following conjecture:

\begin{conj}
\label{thm:conj}
Given a Lagrangian $V \cong \CC\pp^n$ and another Lagrangian $L$ in $\FF \mo$, $\phi_V L$ and $\Phi_V L$ give rise to isomorphic objects in $D\FF \mo$. 
\end{conj}

We stress that a proof of this conjecture would likely require a substantial further analysis: for the parallel argument required to bridge the gap in the spherical case, see \cite{seidelbible}.  We can however verify this conjecture in the case of a $\CC\pp^1$-twist by exploiting the relation 
\begin{equation}
\label{eqn:symp}
\tau_V^2 = \phi_V
\end{equation}
in $Aut^c \mo/Ham \mo$.  Combining this with Seidel's result on spherical twists means that in this case we need only show that $\Phi_V$ and $T^2_V$ give isomorphic functors on $D\FF(M,\o)$.  This is proven in Section \ref{sec:spheres}, where we also show that there exist symplectic manifolds containing a Lagrangian $V$ with $H^*(V) \cong H^*(\CC\pp^n)$ where we can still define $\Phi_V$, but such that this functor has no preimage under \eqref{eqn:symptofuk}, so that $\Phi_V$ has no geometric representative.

Related results to those of this paper of been obtained by Huybrechts and Thomas \cite{huybrechtsthomas}, who introduce the notion of $\pp^n$-objects and $\pp^n$-twist functors for the derived category $D(X)$ of some smooth projective variety $X$.  Our construction is modelled on theirs and our results should be thought of as being ``mirror'' versions.

Returning to symplectic geometry, similar twist maps exist for Lagrangian $\RR\pp^n$s and $\mathbb{H}\pp^n$s since they are compact symmetric spaces of rank one and so admit metrics whose geodesic flow is periodic \cite{besse}.  The results of this paper can easily be reinterpreted in these contexts: we leave it to the interested reader the make the necessary minor adjustments (although we do remark that in the case of $\RR\pp^n$ one has to work in characteristic 2 to avoid sign issues).  The key fact is that $\RR\pp^n$ and $\mathbb{H}\pp^n$ both have cohomology rings which are truncated polynomial algebras (again only in characteristic 2 for $\RR\pp^n$) - indeed, this is necessary for the geodesic flow to be periodic by a theorem of Bott \cite{bott}.

There is an interesting algebraic counterpart to this observation: the construction of \cite{huybrechtsthomas} has been
extended by Grant to give a great many autoequivalences of derived categories \cite{grant}.  He works in the setting of the derived category $D^b(A)$ of modules over a finite dimensional symmetric $k$-algebra $A$ and proves that, given $P$ a projective $A$-module whose endomorphism algebra $\End_A(P)^{op}$ is \emph{periodic}, then there is a related autoequivalence $\Psi_P$ of $D^b(A)$.  This includes the case when the endomorphism algebra is a truncated polynomial ring, but is more general.  It would be interesting to try to understand if there is any geometric motivation for the other autoequivalences that Grant constructs.

\subsection*{Acknowledgements} I would like to thank my PhD supervisor Ivan Smith for suggesting this project and for many helpful discussions.  I would also like to thank Richard Thomas for pointing out several errors in an earlier draft of this paper.  During this research I was partially supported by European Research Council grant ERC-2007-StG-205349. 

\section{\texorpdfstring{$\ainf$-categories}{A-infinity categories}}
\label{sec:ainf}
Here we recall the basic background material on $\ainf$-categories that we
shall need.  Sign conventions differ throughout the literature, but all our
signs and notation come from \cite{seidelbible}, to which we direct the reader who finds the treatment in this section too brief.

\subsection{Categories}
Fix some coefficient field $\KK$.  An $\ainf$-category $\AA$ consists of a set of objects $Ob \AA$ as well as a
finite-dimensional $\ZZ$-graded $\KK$-vector space $\hom_{\AA}(X,Y)$ for any pair of objects $X,Y$, and
composition maps $(\mu_{\AA}^d)_{d\geq 1}$,
\[
\mu_{\AA}^d \colon \hom_{\AA}(X_{d-1}, X_d) \otimes \cdots \otimes
\hom_{\AA}(X_0, X_1)  \to \hom_{\AA}(X_0, X_d)[2-d],
\]
which satisfy the $\ainf$-relations
\begin{equation}
\label{eqn:ainf}
\sum_{m,n} (-1)^{\kreuz_n} \mu_{\AA}^{d-m+1} (a_d, \ldots, a_{n+m+1},
\mu_{\AA}^m (a_{n+m}, \ldots, a_{n+1} ), a_n, \ldots, a_1) =0.
\end{equation}
Here $\kreuz_n = |a_1| + \cdots + |a_n| - n$ and by $[k]$ we mean a shift in
grading \emph{down} by $k$.

The opposite category of $\AA$, denoted $\AA^{opp}$, has the same objects as $\AA$ and
$\hom_{\AA^{opp}}(X,Y) = \hom_{\AA}(Y,X)$, but composition is reversed:
\[
\mu_{\AA^{opp}}^d (a_d, \ldots, a_1) = (-1)^{\kreuz_d} \mu_{\AA}^d (a_1,\ldots, a_d).
\]
The $\ainf$-relations in particular mean that $\mu_{\AA}^1 (\mu_{\AA}^1 (\cdot))=0$
so we can consider the cohomological category $H(\AA)$, which has the same objects as
$\AA$ and has morphism spaces $\hom_{H(\AA)}(X,Y) = H(\hom_{\AA}(X,Y),
\mu_{\AA}^1)$ with (associative) composition
\[
[a_2] \cdot [a_1] = (-1)^{|a_1|} [\mu_{\AA}^2(a_2, a_1)].
\]
We call $\AA$ \emph{cohomologically unital} (c-unital for short) if $H(\AA)$ has
identity morphisms (so is a category in the standard sense).  Although this is
perhaps not the most natural notion in the context of $\ainf$-categories, all
categories considered in this paper will be assumed to be c-unital, since Fukaya categories always carry cohomological units for geometric reasons.

There is another notion of unitality that is helpful to consider although Fukaya categories in general do not satisfy it:  we say $\AA$ is strictly unital if, for each $X$, there is an element $e_X \in \hom^0(X,X)$ such that
\begin{itemize}
\item $\mu^1(e_X) = 0$;
\item $(-1)^{|a|}\mu^2(e_X,a) = a = \mu^2(a, e_X)$ for $a \in \hom(X_0,X_1)$;
\item $\mu^d(a_{d-1}, \ldots, e_X, \ldots, a_1) =0$ for all $d \geq3$.
\end{itemize}
This is useful because every c-unital $\ainf$-category is quasi-equivalent to a strictly unital one \cite[Lemma 2.1]{seidelbible}.

\subsection{Functors}
An $\ainf$-functor $\FF \colon \AA \to \BB$ consists of a map $\FF \colon Ob\AA \to
Ob\BB$ and maps
\[
\FF^d\ \colon \hom_{\AA}(X_{d-1}, X_d) \otimes \cdots \otimes \hom_{\AA}(X_0,
X_1)  \to \hom_{\BB}(\FF X_0, \FF X_d)[1-d]
\]
for all $d \geq1$, which are required to satisfy
\begin{multline}
\label{eqn:functor}
 \sum_r \sum_{s_1 + \cdots + s_r = d} \mu_{\BB}^r (\FF^{s_r}(a_d, \ldots, a_{d- s_r +1}),\ldots,
\FF^{s_1}(a_{s_1}, \ldots, a_1)) \\
 = \sum_{m,n} (-1)^{\kreuz_n} \FF^{d-m+1}(a_d, \ldots,\mu_{\AA}^m (a_{n+m},
\ldots, a_{n+1} ),a_n, \ldots, a_1).
\end{multline}
$\FF$ induces a functor $H\FF \colon H(\AA) \to H(\BB)$ by $[a] \mapsto
[\FF^1(a)]$.  We call a functor $\FF$ between c-unital categories c-unital if
$H\FF$ is unital.  All functors in this paper will be assumed to be c-unital.  We say $\FF$ is \emph{cohomologically full and faithful} if $H\FF$ is full and faithful, and we say $\FF$ is a \emph{quasi-equivalence} if $H\FF$ is an equivalence.

The set of $\ainf$-functors $\FF \colon \AA \to \BB$ can itself be considered as the
objects of an $\ainf$-category $fun(\AA,\BB)$ (or more specifically $nu-fun(\AA,\BB)$ for ``non-unital functors'' if we make no assumptions about units).  We shall only need this in the following
specific context.

\subsection{\texorpdfstring{$\ainf$-modules}{A-infinity modules}}
We first note that any dg category can be considered as an $\ainf$-category
with $\mu^d =0$ for $d \geq 3$.  In particular, for a given $\ainf$-category
$\AA$, we can consider $\ainf$-functors from $\AA^{opp}$ to the category of chain complexes
$Ch$ over $\KK$.  We call such functors $\ainf$-modules over $\AA$.  Such functors can be thought of as the objects of a new $\ainf$-category $\QQ=mod(\AA) =fun(\AA^{opp}, Ch)$.

An $\ainf$-module $\MM \colon \AA \to Ch$ assigns a graded vector space $\MM(X)$ to all
$X \in Ob \AA$ and, in this specific setting, we follow \cite{seidelbible} in changing notation of \eqref{eqn:functor} slightly so that we have maps
\[
\mu_{\MM}^d \colon \MM(X_{d-1}) \otimes \hom_{\AA}(X_{d-2}, X_{d-1}) \otimes
\cdots \otimes \hom_{\AA}(X_0, X_1)  \to \MM(X_0)[2-d]
\]
satisfying
\begin{multline}
\label{eqn:ainfmodules}
\sum_{m,n} (-1)^{\kreuz_n} \mu_{\MM}^{n+1}
(\mu_{\MM}^{d-n}(b,a_{d-1}, \ldots, a_{n+1}),\ldots, a_1) \\
+ \sum_{m,n} (-1)^{\kreuz_n} \mu_{\MM}^{d-m+1}(b, a_{d-1},
\ldots,\mu_{\AA}^m (a_{n+m}, \ldots, a_{n+1}),a_n, \ldots, a_1) = 0.
\end{multline}
The morphism space $\hom^r_{\QQ}(\MM_0, \MM_1)$ in degree $r$ is made up of
so-called \emph{pre-module homomorphisms} $t = (t^d)_{d \geq 1}$ where
\[
t^d \colon \MM_0(X_{d-1}) \otimes \hom_{\AA}(X_{d-2}, X_{d-1}) \otimes \cdots
\otimes \hom_{\AA}(X_0, X_1)  \to \MM_1(X_0)[r - d + 1].
\]
The composition maps in $\QQ$ are
\begin{align}
\label{eqn:mu1q}
& \left( \mu_{\QQ}^1 t \right)^d (b, a_{d-1}, \ldots, a_1) =  \\
& \quad \quad \sum (-1)^{\ddagger} \mu_{\MM_1}^{n+1}( t^{d-n} (b, a_{d-1}, \ldots, a_{n+1}), a_n, \ldots, a_1) \nonumber \\
& \quad \quad {}+ \sum (-1)^{\ddagger} t^{n+1}(\mu_{\MM_0}^{d-n} (b, a_{d-1}, \ldots, a_{n+1}), a_n, \ldots, a_1) \nonumber\\
& \quad \quad {}+ \sum (-1)^{\ddagger} t^{d-m+1} (b, a_{d-1}, \ldots, \mu_{\AA}^m(a_{n+m},\ldots,  a_{n+1}), \ldots, a_1); \nonumber\\
& \left( \mu_{\QQ}^2 (t_2, t_1) \right)^d (b, a_{d-1}, \ldots, a_1) = \\
& \quad \quad \sum (-1)^{\ddagger} t_2^{n+1}( t_1^{d-n} (b, a_{d-1}, \ldots, a_{n+1}), a_n, \ldots, a_1); \nonumber
\end{align}
and $\mu_{\QQ}^d = 0$ for $d \geq 3$.  Here $\ddagger = |a_{n+1}| + \cdots + |a_{d-1}| + |b| - d + n +1$.  We stress that the fact that higher composition maps vanish is not true for more general $\ainf$-functor categories, but rather reflects the dg nature of $Ch$.  

If $\mu^1_{\QQ}t=0$, we say that $t$ is a $\ainf$-\emph{module homomorphism}.  In this situation, we have a map $H(t) \colon H(\MM_0(X)) \to H(\MM_1(X))$ for all $X$, given by $[b] \mapsto [(-1)^{|b|} t^1(b)]$, where here $H(\MM(X))$ is the cohomology of $\MM(X)$ computed with respect to the differential $\partial(b) = (-1)^{|b|}\mu^1_{\MM}(b)$.

\begin{lem}
\emph{(\cite[Lemma 1.16]{seidelbible})}
Suppose the $\ainf$-module homomorphism $t \in \hom_{\QQ}(\MM_0,\MM_1)$ is such
that the induced maps $H(t) \colon H(\MM_0(X)) \to H(\MM_1(X))$ are
isomorphisms for all $X$.  Then, left composition with $t$ induces a
quasi-isomorphism $\hom_{\QQ}(\MM_1,\NN) \to \hom_{\QQ}(\MM_0, \NN)$ and a
similar result holds for right composition.
\end{lem}

\begin{cor}
\label{thm:cohomology}
Under the above hypotheses, $[t]$ is an isomorphism in $H(\QQ)$.
\end{cor}

Given $Y \in \AA$, there is an associated $\ainf$-module $\YY \in \QQ$ where
\[
\YY(X) = \hom_{\AA}(X,Y), \quad  \mu_{\YY}^d=\mu_{\AA}^d.
\]
This forms part of an $\ainf$-functor $\ell \colon \AA \to \QQ$ called the \emph{Yoneda embedding}.  Given $t\in \hom_{\AA}(Y,Z)$, $\ell^1(t) \in \hom_{\QQ}(\YY,\zz)$ is the morphism
\[
\left( \ell^1(t) \right)^d (b, a_{d-1}, \ldots, a_1) = \mu_{\AA}^{d+1}(t,b, a_{d-1}, \ldots, a_1),
\]
and the higher order parts of the functor $\ell$ are defined similarly. $\ell$ is cohomologically full and faithful \cite[Corollary 2.13]{seidelbible}.

\subsection{Twisted complexes}
Given $\AA$ we can form a new category $\Sigma \AA$ called the \emph{additive enlargement} of $\AA$ whose objects are formal sums
\[
X = \bigoplus_{i \in I} V_i \otimes X_i,
\]
where $I$ is some finite set, the $V_i$ are finite-dimensional graded vector spaces and $X_i$ are objects of $\AA$.
\[
\hom_{\Sigma \AA} \left( \bigoplus_{i \in I} V_i \otimes X_i, \bigoplus_{j \in J} W_j \otimes Y_j \right) = \bigoplus_{i,j} \hom_{\KK}(V_i,W_j) \otimes \hom_{\AA}(X_i,Y_j),
\]
and we write morphisms $a\in \hom_{\Sigma\AA}(X,Y)$ as $\alpha^{ji} \otimes x^{ji}$ where $\alpha^{ji}$ and $x^{ji}$ are matrices of morphisms in $\hom_{\KK}(V_i,W_j)$, $\hom_{\AA}(X_i,Y_j)$ respectively.  The composition maps are given by
\[
\mu^d_{\Sigma \AA}(a_d, \ldots , a_1) = \sum (-1)^{\triangleleft}\alpha_d \cdots \alpha_1 \otimes \mu^d_{\AA}(x_d, \ldots, x_1)
\]
where $\triangleleft = \sum_{p<q}|\alpha^{i_p,i_{p-1}}_p|\cdot(|x_q^{i_q,i_{q-1}}|-1)$.
$\AA$ clearly sits inside $\Sigma \AA$ as a full $\ainf$-subcategory once an object $X$ is mapped to $\KK \otimes X$, with $\KK$ given grading zero.

A twisted complex in $\AA$ is an object $X$ of $\Sigma \AA$, together with a differential $\delta_X \in \hom_{\Sigma \AA}^1(X,X)$ which satisfies the following conditions:
\begin{itemize}
\item $\delta_X$ is strictly lower-triangular with respect to some filtration on $X$.  By ``filtration'' here we mean a finite decreasing collection of subcomplexes $F^iX$ such that the induced differential on $F^kX/F^{k+1}X$ is zero \cite[Section 3l]{seidelbible};
\item $\sum_d \mu^d_{\Sigma \AA} (\delta_X, \ldots, \delta_X)=0$.
\end{itemize}
Given this we can define new composition maps
\begin{multline*}
\mu_{Tw\AA}^d(a_d, \ldots, a_1)\\
 = \sum_{i_0, \ldots, i_d} \mu_{\Sigma \AA}^{d+i_0 + \cdots +i_d} \left( \underbrace{\delta_{X_d}, \ldots \delta_{X_d}}_{i_d}, a_d, \underbrace{\delta_{X_{d-1}}, \ldots, \delta_{X_{d-1}}}_{i_{d-1}}, a_{d-1}, \ldots, a_1, \underbrace{\delta_{X_0}, \ldots, \delta_{X_0}}_{i_0} \right).
\end{multline*}
The sum is taken over all $i_j \geq0$, but the conditions on $\delta_X$ imply that this is a finite sum and that moreover the $\ainf$-relations \eqref{eqn:ainf} still hold.  $\Sigma \AA$ sits inside $Tw\AA$ as a full $\ainf$-subcategory given by those twisted complexes with zero differential.

We may relate $Tw\AA$ and $\QQ$ using the diagram below.  $\mathcal{I}$ is the obvious inclusion functor and $\mathcal{I}^*$ is the induced pullback.  The reader may find the appropriate formulae in \cite{seidelbible}.
\begin{equation}
\label{eqn:twq}
\xymatrix{
{\AA} \ar[r]^{\ell} \ar[d]^{\mathcal{I}} & {\QQ} \\
{Tw\AA} \ar[r]^{\ell}  & {mod(Tw\AA).} \ar[u]^{\mathcal{I}^*}
}
\end{equation}
We shall denote the resulting map from $Tw\AA$ by $\QQ$ by $\tilde{\ell}$.

\subsection{Tensor products and shifts}
Working in the larger categories $Tw\AA$ and $\QQ$ allows us perform many familiar algebraic constructions not necessarily possible in $\AA$.  As an example, take a chain complex $(Z,\partial)$ and an $\ainf$-module $\MM \in \QQ$ and define a new $\ainf$-module $Z \otimes \MM \in \QQ$ by
\begin{align}
\label{eqn:tensor}
& (Z \otimes \MM)(X) = Z \otimes \MM(X), \\
& \mu^1_{Z \otimes \MM}(z \otimes b) = (-1)^{|b| - 1} \partial (z) \otimes b +
z \otimes \mu^1_{\MM}(b), \nonumber \\
& \mu^d_{Z \otimes \MM}(z \otimes b, a_{d-1}, \ldots, a_1) = z \otimes
\mu^d_{\MM}(b, a_{d-1}, \ldots, a_1) \text{ for $d \geq 2$.} \nonumber
\end{align}

As a special case of this, consider $Z=\mathbb{K}$, a one-dimensional chain
complex concentrated in degree $-1$ and with trivial differential.  We shall
denote $Z \otimes \MM$ by $S\MM$ and call it the \emph{shift} of $\MM$.  Similarly we have $S^{\sigma}\MM$ for any $\sigma \in \ZZ$ and we have a canonical isomorphism
\[
\hom_{H(\QQ)}(\YY, S^{\sigma}\zz) = \hom_{H(\QQ)}(\YY,\zz)[\sigma].
\]
When $\AA$ is strictly unital, we can do a similar thing with twisted complexes.  Given $(X, \delta_X)\in Tw\AA$ and a chain complex $(Z, \partial)$, we can form the twisted complex
\[
\left( Z \otimes X, \id \otimes \delta_X + \widetilde{\partial} \otimes e_X \right),
\]
where $\widetilde{\partial}(z) = (-1)^{|z|-1} \partial(z)$.  We can also do shifts here: $S^{\sigma}Y = \KK[\sigma] \otimes Y$

\begin{rmk}
\emph{(\cite[Remark 3.2]{seidelbible})}
\label{rmk:strict}
Given a chain complex $(Z,\partial)$, we can form a new chain complex given by $H(Z)$ with trivial differential.  By choosing a linear map that picks a chain representative for each cohomology class, we can define a map $H(Z) \otimes \MM \to Z \otimes\MM$ and Corollary \ref{thm:cohomology} says that this will in fact induce an isomorphism in $H(\QQ)$.
\end{rmk}
%

\subsection{Evaluation maps}
Given $V \in \AA$ and $\YY \in \QQ$ we have an evaluation morphism
\begin{align}
\label{eqn:evaluate}
& ev \colon \YY(V) \otimes \VV \to \YY, \nonumber \\
& ev^d(y \otimes v,a_{d-1}, \ldots , a_1) = \mu^{d+1}_{\YY}(y,v,a_{d-1}, \ldots , a_1).
\end{align}
%
In the strictly unital case, we can also define this for twisted complexes.  In order to define \mbox{$ev \colon \hom_{Tw\AA}(V,Y) \otimes V \to Y$}, we require that $ev$ be an element of $\hom_{Tw\AA}(V,Y)^{\vee} \otimes \hom_{Tw\AA}(V,Y)$.  To do this, choose a homogeneous basis $\{b_i\}$ of $\hom_{Tw\AA}(Y,V)$ and let $\{\beta_i\}$ be the dual basis.  Now let $ev = \sum \beta_i \otimes b_i$.  It is easy to verify that the two maps correspond under $\tilde{\ell}$, so we shall feel justified in abusing notation and referring to both as $ev$ since it will always be clear in which setting we are working.

We can also define a dual evaluation map $ev^{\vee} \colon Y \to \hom_{Tw\AA}(Y,V)^{\vee} \otimes V$ given by $ev^{\vee} = \sum \gamma_j \otimes c_j$ where again $\{c_j\}$ is a basis for $\hom_{Tw\AA}(Y,V)$ and $\{\gamma_i\}$ is the dual basis.

\subsection{Cones and triangles}
Given $t \colon \MM_0 \to \MM_1$ a degree zero module homomorphism, we can form
the mapping cone $\mathcal{C} =Cone(t)$ given by
\begin{align}
\label{eqn:cone}
& \mathcal{C}(X) = \MM_0(X) [1] \oplus \MM_1(X), \nonumber \\
& \mu^d_{\mathcal{C}} \left( \left( \begin{array}{cc} b_0 \\ b_1 \end{array}
\right), a_{d-1}, \ldots , a_1 \right) = \left( \begin{array}{cc}
\mu^d_{\MM_0}(b_0, a_{d-1}, \ldots , a_1) \\ \mu^d_{\MM_1}(b_1, a_{d-1}, \ldots
, a_1) +  t^d(b_0,a_{d-1}, \ldots , a_1) \end{array} \right).
\end{align}

The cone $\cc$ comes with module homomorphisms $\iota$ and $\pi$ which
fit into the following diagram in $H(\QQ)$
\[
\xymatrix{
{\MM_0} \ar[r]^{[t]} & {\MM_1} \ar[d]^{[\iota]}  \\
 & {\mathcal{C}.} \ar[ul]^{[\pi]}_{[1]}
}
\]
Any triangle in $H(\AA)$ quasi-isomorphic to one of the above form under the Yoneda embedding is called \emph{exact}.

Likewise the cone of $t \colon X \to Y$ in $Tw\AA$ for a degree zero cocycle $t$ is given by
\[
Cone(t) = \left(SX \oplus Y, \begin{pmatrix} S(\delta_X) & 0 \\ -S(t) & \delta_Y \end{pmatrix} \right).
\]
We call an $\ainf$-category $\AA$ triangulated if every morphism $[t]$ fits into
some exact triangle and $\AA$ is closed under all shifts, positive and negative.

\begin{prop}
\emph{(\cite[Proposition 3.14]{seidelbible})}
If $\AA$ is a triangulated $\ainf$-category, then $H^0(\AA)$ is triangulated in
the classical sense.  Moreover, for $\FF$ an $\ainf$-functor between
triangulated $\ainf$-categories, $H\FF$ is an exact functor of triangulated
categories.
\end{prop}

For a given $\AA$, we can consider the triangulated $\ainf$-subcategory $\widetilde{\QQ} \subset \QQ$ generated by the image of the Yoneda embedding.  We call $H^0(\widetilde{\QQ})$ the \emph{derived category} of $\AA$, which we denote $D\AA$.  Equivalently, we may define $D\AA$ as $H^0(Tw\AA)$.

\section{\texorpdfstring{$\CC\pp^{\lowercase{n}}$-twists}{CPn twists}}
\label{sec:twists}
In the interests of legibility, we introduce the shorthand $\mathbf{a}_{d-1}$
for $a_{d-1}, \ldots, a_1$.

Huybrechts and Thomas \cite{huybrechtsthomas}, motivated by mirror symmetry, introduced the notion of a $\pp^n$-object $P$ in
the derived category $D(X)$ of a smooth projective variety $X$.  They showed that
there are associated twists $\Phi_P$ of $D(X)$ which are in fact autoequivalences.  We reinterpret their construction in our setting.

\begin{defn}
A $\CC\pp^n$-object is a pair $(V,h)$ where $V\in Ob \AA$ and $h \in \hom^2(V,V)$ such that
\begin{itemize}
\item $\mu_{\AA}^1 h=0$;
\item $\hom_{H(\AA)}(V,V) \cong \KK[h]/h^{n+1}$ as a graded ring;
\item There exists a map $\int \colon \hom^{2n}_{H(\AA)}(V,V) \to \mathbb{K}$ such that, for any $X$, the resulting bilinear map $\hom^{2n-k}_{H(\AA)}(X,V) \times \hom^k_{H(\AA)}(V,X) \to \hom^{2n}_{H(\AA)}(V,V) \to \mathbb{K}$ is nondegenerate.
\end{itemize}
We shall often just refer to a $\CC\pp^n$-object by $V$ since, following Remark \ref{rmk:strict}, the choice of $h$ will be irrelevant up to quasi-equivalence.
\end{defn}

To define our twist, we imitate the construction in \cite{huybrechtsthomas}.  Take some $\CC\pp^n$-object $V$ and consider the following diagram
\begin{equation}
\label{eqn:phi}
\xymatrix{
{\YY(V)[-2] \otimes \VV} \ar[r]^{H} & {\YY(V) \otimes \VV} \ar[r]^{\iota}
\ar[dr]^{ev} & {\HH_{\YY}} \ar[d]^{g}  \\
                   &                               & {\YY} \ar[d] \\
                   &                               & {\Phi_V\YY}
}
\end{equation}
where here $\HH_{\YY}$ is $Cone(H)$ and $\Phi_V\YY$ is $Cone(g)$.

Here $ev$ is the evaluation map \eqref{eqn:evaluate} and we define the other
maps by
\begin{align*}
& H^1(y \otimes v) = (-1)^{|y|+|v|} \mu_{\YY}^2(y,h) \otimes v + (-1)^{|y|-1} y \otimes \mu_{\VV}^2(h,v),\\
& H^d(y\otimes v,\mathbf{a}_{d-1}) = (-1)^{|y|-1}y \otimes \mu_{\VV}^{d+1}(h, v,\mathbf{a}_{d-1}) \text{\ for $d\geq2$}.
\end{align*}
and
\begin{align*}
& g^d \left( \left({ \begin{array}{c} y_1 \otimes v_1 \\y_2 \otimes v_2 \end{array} } \right),\mathbf{a}_{d-1} \right) =  \mu_{\YY}^{d+1} \left(y_2, v_2,\mathbf{a}_{d-1} \right) + (-1)^{|y_1|-1}\mu_{\YY}^{d+2} \left(y_1, h, v_1,\mathbf{a}_{d-1} \right). 
\end{align*}

\begin{lem}
$H$ and $g$ are $\mu^1_{\QQ}$-closed.
\end{lem}

\begin{proof}
This is a direct calculation.  Using \eqref{eqn:mu1q} and \eqref{eqn:tensor},
we see that
\begin{align*}
&\left( \mu^1_{\QQ} H \right)^d(y \otimes v, \mathbf{a}_{d-1}) = \\
& \;y \otimes \left( \begin{array}{c} 
\sum_n  (-1)^{\ddagger_n+|y|-1}\mu_{\VV}^{n+1}(\mu^{d-n+1} (h,v , a_{d-1}, \ldots, a_{n+1}), \ldots, a_1) \\
{} + \sum_n (-1)^{\ddagger_n+|y|-1}\mu_{\VV}^{n+2}(h, \mu^{d-n}(v,a_{d-1}, \ldots, a_{n+1}), \ldots, a_1) \\
{} + \sum_{m,n} (-1)^{\ddagger_n+|y|-1}\mu_{\VV}^{d-m+2}(h,v,a_{d-1}, \ldots, \mu_{\AA}^m( a_{n+m}, \ldots, a_{n+1}), \ldots, a_1)   \end{array} \right) \\
& \;{} + \left( (-1)^{\ddagger_0 + |y| -1 +|\mu^{d+1}(h,v,\mathbf{a}_{d-1})|-1} + (-1)^{\ddagger_{d-1}+ |v| -1+|y|-2}\right) \mu_{\YY}^1(y) \otimes \mu_{\VV}^{d+1}(h,v,\mathbf{a}_{d-1})\\
& \;{} + \left((-1)^{\ddagger_{d-1}+|y|+|v|} + (-1)^{\ddagger_0+|y|+|\mu^d(v,\textbf{a}_{d-1})|}\right)\mu_{\YY}^2(y,h) \otimes \mu_{\VV}^d(v,\mathbf{a}_{d-1} ). 
\end{align*}

The terms involving $\mu_{\YY}^1(y)$ and $\mu_{\YY}^2(y,h)$ cancel, and inside the big
bracket, we find precisely the terms from the $\ainf$-relation \eqref{eqn:ainf} except for the term involving $\mu_{\VV}^d(\mu_{\AA}^1(h),v, \textbf{a}_{d-1})$.  But, by assumption, $\mu_{\AA}^1(h)=0$ so this term vanishes.

The proof for $g$ is similar:
\begin{align*}
&\left( \mu^1_{\QQ} g \right)^d \left( \left({ \begin{array}{c} y_1 \otimes v_1 \\y_2 \otimes v_2 \end{array} }\right),\mathbf{a}_{d-1} \right) = \\
& {}\phantom{{}+{}} \sum_n  (-1)^{\ddagger_n+|y_1|-1} \mu_{\YY}^{n+1}(\mu_{\YY}^{d-n+2} (y_1,h,v_1 , a_{d-1}, \ldots, a_{n+1}), \ldots, a_1) \\
& {} + \sum_n  (-1)^{\ddagger_n+|y_1|-1} \mu_{\YY}^{n+3}(y_1,h, \mu_{\VV}^{d-n}(v_1,a_{d-1}, \ldots, a_{n+1}), \ldots, a_1) \\
& \quad \quad \quad {} + (-1)^{\ddagger_{d-1} + |\mu^1(y_1)|-1 +|v_1|-1} \mu_{\YY}^{d+2}(\mu_{\YY}^1(y_1),h,v_1, a_{d-1}, \ldots, a_1) \\
& {} + \sum_{m,n} (-1)^{\ddagger_n+|y_1|-1} \mu_{\YY}^{d-m+3}(y_1,h,v_1,a_{d-1}, \ldots, \mu_{\AA}^m( a_{n+m}, \ldots, a_{n+1}), \ldots, a_1) \\
& {} \\
& {} + \sum_n  (-1)^{\ddagger_n}\mu_{\YY}^{n+1}(\mu_{\YY}^{d-n+1} (y_2,v_2 , a_{d-1}, \ldots, a_{n+1}), \ldots, a_1) \\
& {} + \sum_n (-1)^{\ddagger_n}\mu_{\YY}^{n+2}(y_2, \mu_{\VV}^{d-n}(v_2,a_{d-1}, \ldots, a_{n+1}), \ldots, a_1) \\
& \quad \quad \quad {} + (-1)^{\ddagger_{d-1} + |v_2|-1} \mu_{\YY}^{d+1}(\mu_{\YY}^1(y_2),v_2, a_{d-1}, \ldots, a_1) \\
& {} + \sum_{m,n} (-1)^{\ddagger_n}\mu_{\YY}^{d-m+2}(y_2,v_2,a_{d-1}, \ldots, \mu_{\AA}^m( a_{n+m}, \ldots, a_{n+1}), \ldots, a_1) \\
& {} \\
& {} + \sum (-1)^{\ddagger_n+|y_1|-1}\mu_{\YY}^{n+2}(y_1, \mu_{\VV}^{d-n+1}(h,v_1,a_{d-1}, \ldots, a_{n+1}), \ldots, a_1) \\
& \quad \quad \quad {} + (-1)^{\ddagger_{d-1} + |y_1| + |v_1|} \mu_{\YY}^{d+1}(\mu_{\YY}^2(y_1,h),v_1, a_{d-1}, \ldots, a_1).
\end{align*}

Here the final two lines come from the presence of $H$ in the $\mu^d$ maps in $Cone(H)$ as in \eqref{eqn:cone}.  Again we find all the terms from \eqref{eqn:ainfmodules} except for those involving $\mu_{\AA}^1(h)$, so the above sum vanishes.

\end{proof}

Concretely, $\Phi_V\YY = \left(\YY(V) \otimes \VV \right) \oplus
\left(\YY(V)[1] \otimes \VV \right) \oplus \YY$ and
\begin{equation*}
\mu_{\Phi_V\YY}^1 \left( \begin{array}{ccc}
y_1 \otimes v_1\\
\begin{array}{cc}y_2 \otimes v_2 \\ {\;} \end{array}\\
y_3 \end{array}
\right) 
=
\left( \begin{array}{ccc}
(-1)^{|v_1|-1}\mu_{\YY}^1(y_1) \otimes v_1      +      y_1 \otimes \mu_{\VV}^1(v_1) \\
\begin{array}{cc} (-1)^{|v_2|-1}\mu_{\YY}^1(y_2) \otimes v_2      +   (-1)^{|y_1|+|v_1|}\mu_{\YY}^2(y_1,h) \otimes v_1    \\+ y_2 \otimes \mu_{\VV}^1(v_2) + (-1)^{|y_1|-1}y_1 \otimes \mu_{\VV}^2(h,v_1) \end{array} \\
\mu_{\YY}^1(y_3) + \mu_{\YY}^2(y_2,v_2) +  (-1)^{|y_1|-1}\mu_{\YY}^3(y_1, h, v_1) \end{array}
\right)
\end{equation*}
and, for $d\geq2$,
\begin{multline*}
\mu_{\Phi_V\YY}^d \left( \left( \begin{array}{ccc}
y_1 \otimes v_1\\
y_2 \otimes v_2\\
y_3 \end{array}
\right), \mathbf{a}_{d-1} \right) \\
=
\left( \begin{array}{ccc}
y_1 \otimes \mu_{\VV}^d(v_1,\mathbf{a}_{d-1}) \\
y_2 \otimes \mu_{\VV}^d(v_2,\mathbf{a}_{d-1}) + (-1)^{|y_1|-1}y_1\otimes \mu_{\VV}^{d+1}(h,v_1, \mathbf{a}_{d-1}) \\
\mu_{\YY}^d(y_3,\mathbf{a}_{d-1}) + \mu_{\YY}^{d+1}(y_2,v_2,\mathbf{a}_{d-1}) + (-1)^{|y_1|-1}
\mu_{\YY}^{d+2}(y_1, h, v_1,\mathbf{a}_{d-1}) \end{array}
\right).
\end{multline*}

\subsection{\texorpdfstring{$\CC\pp^n$-twist functor}{CPn twist functor}}
We want to upgrade $\Phi_V$ to a functor $\Phi_V \colon \QQ \to \QQ$ and so,
having described the effect of $\Phi_V$ on objects, we must describe how it
acts on morphisms.

Firstly we set $\Phi_V^d=0$ for $d \geq 2$, so that $\Phi_V$ is in fact a dg
functor and, given $t \in \hom_{\QQ}(\YY,\zz)$, $\hat{t} = \Phi_V(t)$
has first order part
\[
\hat{t}^1 \left( \begin{array}{ccc}
y_1 \otimes v_1\\
y_2 \otimes v_2\\
y_3 \end{array}
\right)
=
\left( \begin{array}{ccc}
(-1)^{|v_1|+|t|} t^1(y_1) \otimes v_1 \\
(-1)^{|v_2|-1} t^1(y_2) \otimes v_2 + (-1)^{|y_1|+|v_1|}t^2(y_1,h) \otimes v_1 \\
t^1(y_3) + t^2(y_2,v_2) +  (-1)^{|y_1|-1} t^3(y_1, h, v_1) \end{array}
\right)
\]
and, for $d\geq2$, 
\begin{multline*}
\hat{t}^d \left( \left( \begin{array}{ccc}
y_1 \otimes v_1\\
y_2 \otimes v_2\\
y_3 \end{array}
\right), \mathbf{a}_{d-1} \right) \\
=
\left( \begin{array}{ccc}
0 \\
0 \\
t^d(y_3, \mathbf{a}_{d-1}) + t^{d+1}(y_2,v_2, \mathbf{a}_{d-1}) + (-1)^{|y_1|-1} t^{d+2}(y_1, h, v_1, \mathbf{a}_{d-1}) \end{array}
\right).
\end{multline*}

\begin{lem}
\label{thm:functor}
$\Phi_V$ is an $\ainf$-functor.
\end{lem}

\begin{proof}
The condition we need to verify is \eqref{eqn:functor}, which here reduces to
the two conditions
\begin{align*}
\mu^1_{\QQ} (\Phi^1_V(t_1)) & = \Phi^1_V( \mu^1_{\QQ} (t_1)), \\
\mu^2_{\QQ} (\Phi^1_V(t_2), \Phi^1_V(t_1)) & = \Phi^1_V( \mu^2_{\QQ} (t_2, t_1)),
\end{align*}
since $\mu^d_{\QQ}=0$ for $d \geq 3$. Both are straightforward calculations.
\end{proof}

\begin{prop}
\label{thm:shift}
$\Phi_V \VV \cong S^{-2n}\VV$.   Also, if $\hom_{H(\AA)}(V,Y)=0$, then $\Phi_VY \cong Y$.
\end{prop}

We first recall a basic algebraic lemma that we shall need.

\begin{lem}
\label{thm:quis}
If $f \colon V \to W$ is a map of chain complexes such that $f$ is surjective
and $\ker f$ is acyclic, then $f$ is a quasi-isomorphism.
\end{lem}
\begin{proof}[Proof of Proposition \ref{thm:shift}]
Following Remark \ref{rmk:strict}, we may replace the $\hom_{\AA}(V,Y)$ terms in $\Phi_V\YY$ with $\hom_{H(\AA)}(V,Y)$.  This vector space has a basis given by $e_V, h, \ldots, h^n$ so that we replace $\Phi_V \VV$ with the quasi-isomorphic
\[
\bigoplus_{i=0}^n h^i[-2i] \VV \oplus \bigoplus_{i=0}^n h^i[-2i+1] \VV \oplus \VV.
\]
There is a module homomorphism $\pi_1$ that to first-order is a projection annihilating the summands $e_V[1] \VV \oplus \VV$ and has higher-order terms zero.  We want to apply Corollary $\ref{thm:cohomology}$ to $\pi_1$, so let $\partial(b) = (-1)^{|b|}\mu^1_{\Phi_V\YY}(b)$.  Now, if an element $(0,e_V \otimes v_2, v_3)$ of the kernel of $\pi_1$ is $\partial$-closed, then
$\partial (0,-e_V \otimes v_3,0)=(0,e_V \otimes v_2, v_3)$, so by Lemma \ref{thm:quis}, $\pi_1$ is a quasi-isomorphism.

This means $\Phi_V \VV$ is quasi-isomorphic to the image of $\pi_1$,
\[
\bigoplus_{i=0}^n h^i[-2i] \VV \oplus \bigoplus_{i=1}^n h^i[-2i+1] \VV. 
\]
We can project once more so as to kill the summands $e_V \VV \oplus h[-1]\VV$.  A similar argument shows
that this is a quasi-isomorphism.  By repeating this process, removing pairs of summands by a series of projection quasi-isomorphisms, we can remove everything except $h^n[-2n] \VV$.

The second fact is trivial from the definition of $\Phi_V$.
\end{proof}

\begin{rmk}
These results coincide with what one finds geometrically: namely that $\phi_V$ acts on itself by a shift in grading by $2n$ \cite{seidelgradings}, and if $W$ and $V$ are disjoint Lagrangians, then we can arrange that $\phi_V$ is supported in a region disjoint from $W$ so that $\phi_V$ has no effect on $W$.
\end{rmk}

\section{\texorpdfstring{$\Phi_V$ is a quasi-equivalence}{Phi is a quasi-equivalence}}
\label{sec:equivalence}
In the case where $\AA$ itself is a triangulated $\ainf$-category, the discussion in \cite[Section 3d]{seidelbible} shows that we can define an $\ainf$-functor, which we shall also denote $\Phi_V$, on $\AA$ itself in such a way that
\begin{equation}
\label{eqn:pullback}
\xymatrix{
{\AA} \ar[r]^{\ell} \ar[d]^{\Phi_V} & {\QQ} \ar[d]^{\Phi_V} \\
{\AA} \ar[r]^{\ell}  & {\QQ}
}
\end{equation}
commutes (up to isomorphism in $H^0(fun(\AA,\QQ))$).  In this section we shall prove
\begin{thm}
\label{thm:pnmain}
If $V$ is a $\CC\pp^n$-object in a cohomologically finite $\ainf$-triangulated category $\AA$, then $\Phi_V \colon
\AA \to \AA$ is a quasi-equivalence.
\end{thm}

To prove this, it will be useful to have an explicit formula for $\Phi_V$ on the level of twisted complexes.  In order to do this, we have to assume that $\AA$ is strictly unital.  The more general c-unital case later will be discussed later.  The diagram \eqref{eqn:pullback} can now be augmented to the following:
\begin{equation}
\label{eqn:pullback+}
\xymatrix{
{\AA} \ar[r] \ar@{.>}[d]^{\Phi_V} & {Tw\AA} \ar[d]^{\Phi_V} \ar[r]^{\tilde{\ell}} & {\QQ} \ar[d]^{\Phi_V} \\
{\AA} \ar[r]               & {Tw\AA} \ar[r]^{\tilde{\ell}}                   & {\QQ.}
}
\end{equation}
We shall define an $\ainf$-functor $\Phi_V$ on $Tw\AA$ such that the righthand square precisely commutes.  Then, in the case where $\AA$ is triangulated, the inclusion into $Tw\AA$ is a quasi-equivalence so can be inverted \cite[Theorem 2.9]{seidelbible}, which allows us to pullback $\Phi_V$ to $\AA$.

In order to imitate the construction of Section \ref{sec:twists} we need to define a map $H \colon \hom_{Tw\AA}(V,Y) \otimes V \to \hom_{Tw\AA}(V,Y) \otimes V$.  This means $H$ must be an element of $\End_{\KK}(\hom_{Tw\AA}(V,Y)) \otimes \hom_{Tw\AA}(V,V)$.  Let $\bar{h} \in \End_{\KK}(\hom_{Tw\AA}(V,Y))$ be the linear map $a \mapsto \mu^2(a,h)$, and define
\[
H = \bar{h} \otimes e_V - \id \otimes h.
\]
With this we can consider the diagram
\begin{equation}
\label{eqn:twistedphi}
\xymatrix{
{\hom_{Tw\AA}(V,Y)[-2] \otimes V} \ar[r]^{H} & {\hom_{Tw\AA}(V,Y) \otimes V} \ar[r]^{\:\:\:\:\:\:\:\:\:\:\:\:\:\:\:\iota}
\ar[dr]^{ev} & {\HH_{Y}} \ar[d]^{g}  \\
                   &                               & {Y} \ar[d] \\
                   &                               & {\Phi_VY.}
}
\end{equation}
As in \eqref{eqn:phi}, $\mathcal{H}_Y=Cone(H)$ and $\Phi_VY=Cone(g)$, where now $g$ is now given by $ev$ on the second summand of $\HH_Y$ and zero on the first summand.  It is straightforward to verify that the above diagram becomes \eqref{eqn:phi} under $\tilde{\ell}$.  We have now defined a twisted complex
\begin{equation}
\label{eqn:twistedcxphi}
\Phi_V Y = \left( \begin{aligned}  &\hom_{Tw\AA}(V,Y) \otimes V \\  &\oplus \hom_{Tw\AA}(V,Y)[1] \otimes V \\  & \oplus Y \end{aligned}\:
, \begin{pmatrix}
  \delta_{\hom_{Tw\AA}(V,Y)\otimes V} & 0 & 0 \\
  -S^2(H) & -\delta_{\hom_{Tw\AA}(V,Y)\otimes V} & 0 \\
  0 & -S(ev) & \delta_Y
\end{pmatrix}
\right).
\end{equation}
Also, given $t \in \hom_{Tw\AA}(Y,Z)$, we get $\Phi_Vt \in \hom_{Tw\AA}(\Phi_VY, \Phi_VZ)$ given with respect to the above splittings by
\[
\begin{pmatrix}
  (-1)^{|t|}\bar{t}\otimes e_V & 0 & 0\\
  \accentset{\triangle}{t}\otimes e_V & \bar{t}\otimes e_V & 0 \\
  0 & 0 & t
\end{pmatrix},
\]
where $\bar{t} \colon \hom_{Tw\AA}(V,Y) \to \hom_{Tw\AA}(V,Z)$ is given by $a \mapsto (-1)^{|a|}\mu^2(t,a)$ and $\accentset{\triangle}{t}$ denotes the map $a \mapsto \mu^3(t,a,h)$.  This now defines an $\ainf$-functor $\Phi_V$ on $Tw\AA$ which has only first-order terms (it is a dg functor).  We leave it to the reader to check that the righthand square in \eqref{eqn:pullback+} commutes.

\subsection{Adjoints}
One of the benefits of making the assumption of strict unitality and working with twisted complexes is that is easy now to identify an adjoint twist functor to $\Phi_V$.  We recall that, given a pair of functors $F: \mathcal{D} \rightarrow \mathcal{C}$ and $G: \mathcal{C} \rightarrow \mathcal{D}$, we say that $F$ is left adjoint to $G$ (and $G$ is right adjoint to $F$) if there are isomorphisms $\hom_{\mathcal{C}}(FY,X) \cong \hom_{\mathcal{D}}(Y,GX)$ which are natural in $X$ and $Y$. 
    
Consider the following diagram
\begin{eqnarray}
\label{eqn:adjoint}
\xymatrix{
& & S^{-1}Y \ar[dl]^{ev ^{\vee}} \ar[d]^{g^{\vee}} \\
{\hom_{Tw\AA}(Y,V)[1]^{\vee} \otimes V} \ar[r]^{H^{\vee}} & {\hom_{Tw\AA}(Y,V)[-1]^{\vee} \otimes V}
\ar[r]^{\:\:\:\:\:\:\:\:\:\:\iota}  & {Cone(H^{\vee})} \ar[d]  \\
                   &                               & {Cone(g^{\vee})}
}
\end{eqnarray}
Define $h^{\vee} \colon \hom_{Tw\AA}(Y,V)[-2]^{\vee} \to \hom_{Tw\AA}(Y,V)^{\vee}$ by $h^{\vee}(\eta)(a)=\eta(\mu^2(h,y))$. Now let $H^{\vee}=h^{\vee}\otimes e_V - \id \otimes h$ and $g^{\vee}=(0,ev^{\vee})$.  We define $\HH^{\vee}_Y= Cone(H^{\vee})$ and $\Phi^{\vee}_V Y = Cone(g^{\vee})$.  $\Phi^{\vee}_V Y$ is given by the twisted complex 
\[
\left( \begin{aligned} &Y  \\ & \oplus \hom_{Tw\AA}(Y,V)[2]^{\vee} \otimes V \\ & \oplus \hom_{Tw\AA}(Y,V)[-1]^{\vee} \otimes V \end{aligned} \:,
\begin{pmatrix}
 \delta_Y & 0 & 0 \\
 0&  \delta_{\hom_{Tw\AA}(V,Y)^{\vee} \otimes V} & 0  \\
  ev^{\vee} & H^{\vee} & \delta_{\hom_{Tw\AA}(V,Y)^{\vee} \otimes V}
\end{pmatrix}\right).
\]
Given  $t \in \hom_{Tw\AA}(Y,Z)$, we similarly get $\Phi^{\vee}_Vt \in \hom_{Tw\AA}(\Phi^{\vee}_VY, \Phi^{\vee}_VZ)$, so that $\Phi^{\vee}_V$ is a (dg) functor on $Tw\AA$.

\begin{prop}
$H\Phi_V^{\vee}$ is both left and right adjoint to $H\Phi_V$.
\end{prop}

\begin{proof}
We first prove that $H\Phi_V^{\vee}$ is left adjoint to $H\Phi_V$.  We want to show there are isomorphisms
\[
\hom_{D\AA}(\Phi^{\vee}_V Y,Z) \cong \hom_{D\AA}(Y, \Phi_V Z)
\]
that are natural in $D\AA$.  By applying the exact functors $\hom_{D\AA} (-,Z)$ to \eqref{eqn:adjoint} and $\hom_{D\AA} (Y,-)$ to \eqref{eqn:twistedphi}, we get long exact sequences, natural in $D\AA$,
\[
\xymatrix@C=-3.5em{
{\hom_{D\AA}(\HH^{\vee}_Y,Z)}
\ar[rr]
\ar@{.>}[dd]
&& {\hom_{D\AA}(\hom_{D\AA}(Y,V)[-1]^{\vee} \otimes V,Z)}
\ar[dl] \ar[dd] \\
& {\hom_{D\AA}(\hom_{D\AA}(Y,V)[1]^{\vee} \otimes V,Z)} \ar[ul]^{[1]} \ar[dd] \\ 
{\hom_{D\AA}(Y,S\HH_Z)}
\ar'[r][rr]
&& {\hom_{D\AA}(Y,\hom_{D\AA}(V,Z)[-1] \otimes V)}
\ar[dl] \\
& {\hom_{D\AA}(Y,\hom_{D\AA}(V,Z)[1]\otimes V)} \ar[ul]^{[1]}  \\ }
\]
%
Here the vertical isomorphisms come from the natural identities
\begin{align*}
\hom_{D\AA}(\hom_{D\AA}(Y,V)^{\vee} \otimes V,Z) &= \hom_{D\AA}(Y,V)^{\vee \vee} \otimes \hom_{D\AA}(V,Z)\\
&= \hom_{D\AA}(Y,V) \otimes \hom_{D\AA}(V,Z)\\
&= \hom_{D\AA}(Y, \hom_{D\AA}(V,Z) \otimes V)
\end{align*}
so that $\hom_{D\AA}(\HH^{\vee}_Y,Z) \cong \hom_{D\AA}(Y,\HH_Z)$ naturally (note that this requires that $\AA$ be cohomologically finite).  This proves in particular that the functor assigning $Y$ to  $\mathcal{H}^{\vee}_Y$ is left adjoint to the functor sending $Y$ to $\mathcal{H}_Y$ (these functors are defined by the obvious restriction of the above construction).

Similarly we have
\[
\xymatrix@C=-1.5em{
{\hom_{D\AA}(\Phi^{\vee}_V Y,Z)}
\ar[rr]
\ar@{.>}[dd]
&& {\hom_{D\AA}(\HH_Y^{\vee},Z)}
\ar[dl] \ar[dd] \\
& {\hom_{D\AA}(S^{-1}Y,Z)} \ar[ul]^{[1]} \ar[dd] \\ 
{\hom_{D\AA}(Y, \Phi_V Z)}
\ar'[r][rr]
&& {\hom_{D\AA}(Y,S\HH_Z)}
\ar[dl] \\
& {\hom_{D\AA}(Y,SZ)} \ar[ul]^{[1]}  \\ }
\]
and therefore
\[
\hom_{D\AA}(\Phi^{\vee}_V Y,Z) \cong \hom_{D\AA}(Y, \Phi_V Z)
\]
naturally.  Proving right adjointness is similar.
\end{proof}

With the existence of adjoints proven, the rest of the proof of Theorem \ref{thm:pnmain} is an exercise in the abstract machinery of triangulated categories.

\subsection{Spanning classes}
A nontrivial collection $\Omega$ of objects in a triangulated category $\mathcal{D}$ is called a
spanning class if, for all $B \in \mathcal{D}$, we have
\begin{itemize}
\item If $\hom_{\dd}(A,B[i]) =0$ for all $A \in \Omega$ and all $i \in \ZZ$, then $B
\simeq 0$.
\item If $\hom_{\dd}(B[i],A) =0$ for all $A \in \Omega$ and all $i \in \ZZ$, then $B \simeq 0$.
\end{itemize}

Given an object $A \in \dd$, we denote by $A^{\perp} = \{ B : \hom^*_{\dd}(A,B)=0 \}$ and can define $^{\perp}A$ similarly.

\begin{lem}
\label{thm:span}
For a $\CC\pp^n$-object $V \in \AA$, $\{V\} \cup V^{\perp}$ is a spanning class in
$D\AA$.
\end{lem}

\begin{proof}
Suppose we have $B$ such that  $\hom_{D\AA}(A,B[i]) =0$ for all $A \in \Omega$ and all
$i \in \ZZ$.  Then putting $A = V$ shows that $B \in V^{\perp}$.  Therefore, in
particular, $\hom_{D\AA}(B,B[i])=0$ for all $i$ so that $B \simeq 0$.
For the other condition, note that, by the definition of $\CC\pp^n$-object, $\hom_{D\AA}(V,A)=0$, if and only if
$\hom_{D\AA}(A,V)=0$.
\end{proof}

\subsection{Equivalence}
We now appeal to the following theorem of Bridgeland \cite[Theorem 2.3]{bridgeland}

\begin{thm} Let $F \colon \cc \to \dd$ be an exact functor between $\KK$-linear triangulated categories such that $F$ has a left and a right adjoint. Then $F$ is fully faithful if and only if there exists some spanning class $\Omega \subset \cc$ such that, for all objects $K,L \in \Omega$ and all $i \in \ZZ$ the natural homomorphism
\[
F \colon \hom_{\cc}(K,L[i]) \to \hom_{\dd}(F(K),F(L[i]))
\] 
is an isomorphism
\end{thm}

For the spanning class from Lemma \ref{thm:span}, this condition follows immediately from Proposition \ref{thm:shift}, so $\Phi_V$ is cohomologically full and faithful.

To show that it is an quasi-equivalence, let $\BB \subset Tw\AA$ be the full $\ainf$-subcategory of objects isomorphic to $\Phi_VY$ for some $Y$.  Since $\Phi_V$ maps exact triangles in $H(Tw\AA)$ to exact triangles in $H(Tw\AA)$, $\BB$ is actually a triangulated $\ainf$-category.  On the other hand, from Proposition \ref{thm:shift}, $V \in \BB$ and so \eqref{eqn:twistedphi} shows that $\BB$ generates $Tw\AA$.  This means that the inclusion $\BB \to Tw\AA$ must be a quasi-equivalence, which implies that $\Phi_V$ is also a quasi-equivalence.

So far we have only dealt with the case when $\AA$ is strictly unital.  In the c-unital case, the standard trick \cite[Section 2]{seidelbible} is to pass to a quasi-equivalent $\ainf$-category $\widetilde{\AA}$ which is strictly unital and such that
\[
\xymatrixcolsep{.5in}
\xymatrixrowsep{.5in}
\xymatrix{
{D\AA} \ar[r]^{H\Phi_V} \ar[d]^{\cong} & {D\AA} \ar[d]^{\cong} \\
{D\widetilde{\AA}} \ar[r]^{H\widetilde{\Phi_V}}  & {D\widetilde{\AA}}
}
\]
commutes (up to isomorphism).  Then we can apply our result from the strictly unital case to complete the proof of Theorem \ref{thm:pnmain}.

\section{Some geometric consequences}
\label{sec:spheres}
\subsection{The connection with spherical objects}
As we mentioned in the Introduction, it would require a more substantial analysis to verify that $\Phi_V$ does in fact represent the categorical version of $\phi_V$.  However, in the lowest dimension when $V \cong \CC\pp^1$, this can be done by using Seidel's result resulting geometric Dehn twists and algebraic spherical twists \cite{seidelbible}, and the relationship \eqref{eqn:symp}.

We shall first recall the basic facts about spherical objects and spherical twists \cite[Section 5]{seidelbible}.

\begin{defn}
An object $V\in \AA$ is called spherical of dimension $n$ if
\begin{itemize}
\item $\hom_{H(\AA)}(V,V) \cong \mathbb{K}[t]/t^2$.
\item There exists a map $\int \colon \hom^n_{H(\AA)}(V,V) \to \mathbb{K}$ such
that, for all $X$, the resulting bilinear map $\hom^{n-k}_{H(\AA)}(X,V) \times
\hom^k_{H(\AA)}(V,X) \to \hom^n_{H(\AA)}(V,V) \to \mathbb{K}$ is nondegenerate.
\end{itemize}
\end{defn}

\begin{defn}
Given an object $V$, the twist map $T_V$ is defined by $\TT_V \YY = Cone(ev)$.
\end{defn}

This forms part of a functor $T_V \colon \QQ \to \QQ$ where, given $t \in \hom_{\QQ}(\YY,\zz)$, $\tilde{t} = T_V(t)$ has first order part
\[
\tilde{t}^1 \left( \begin{array}{cc}
y_1 \otimes v\\
y_2 \end{array}
\right)
=
\left( \begin{array}{cc}
(-1)^{|v|-1}t^1(y_1) \otimes v \\
t^1(y_2)  + t^2(y_1,v) \end{array}
\right)
\]
and
\begin{eqnarray*}
\tilde{t}^d \left( \left( \begin{array}{cc}
y_1 \otimes v\\
y_2 \end{array}
\right), \mathbf{a}_{d-1} \right)
=
\left( \begin{array}{cc}
0\\
t^d(y_2, \mathbf{a}_{d-1})  + t^{d+1}(y_1,v, \mathbf{a}_{d-1}) \end{array}
\right).
\end{eqnarray*}
If $\AA$ is triangulated, we may define the functor $T_V$ on $\AA$ and Seidel proves the following lemma:
\begin{lem}
\emph{(\cite[Lemma 5.11]{seidelbible})}
Given a spherical object $V$ in a c-finite triangulated $\ainf$-category $\AA$, the spherical twist $T_V$ is a
quasi-equivalence of $\AA$.
\end{lem}

\begin{thm}
When $V$ is a $\CC\pp^1$-object (so is also a spherical object of dimension 2), $T_V^2$ and $\Phi_V$ give rise to isomorphic functors on $D\AA$.
\end{thm}

\begin{proof}
$T_V(T_V \YY) = \left(\YY(V) \otimes \VV(V)[2] \otimes \VV \right) \oplus
\left(\YY(V)[1] \otimes \VV \right) \oplus \left(\YY(V)[1] \otimes \VV \right)
\oplus \YY$ with
\begin{multline*}
\mu_{T^2_V\YY}^1 \left( \begin{array}{cccc}
y_1 \otimes q \otimes v_1\\
y_2 \otimes v_2\\
y_3 \otimes v_3\\
y_4 \end{array}
\right)\\
=
\left( \begin{array}{cccc}
(-1)^{|v_1|+|q|}\mu^1(y_1)\otimes q \otimes v_1 + (-1)^{|v_1|-1}y_1 \otimes \mu^1(q) \otimes v_1 + y_1\otimes q \otimes \mu^1(v_1) \\
(-1)^{|v_2|-1}\mu^1(y_2) \otimes v_2 + y_2 \otimes \mu^1(v_2) + (-1)^{|v_1|-1}\mu^2(y_1, q) \otimes v_1 \\
(-1)^{|v_3|-1}\mu^1(y_3) \otimes v_3 + y_3 \otimes \mu^1(v_3) + y_1\otimes \mu^2(q,v_1)  \\
\mu^1(y_4) + \mu^2(y_2,v_2) + \mu^2(y_3, v_3) + \mu^3(y_1, q, v_1) \end{array}
\right)
\end{multline*}
and
\begin{multline*}
{\mu_{T^2_V\YY}^d \left( \left( \begin{array}{cccc}
y_1 \otimes q \otimes v_1\\
y_2 \otimes v_2\\
y_3 \otimes v_3\\
y_4 \end{array}
\right), \textbf{a}_{d-1} \right)} \\
{= 
\left( \begin{array}{cccc}
y_1 \otimes q \otimes \mu^d(v_1, \textbf{a}_{d-1}) \\
y_2 \otimes \mu^d(v_2, \textbf{a}_{d-1})\\
y_3 \otimes \mu^d(v_3, \textbf{a}_{d-1}) + y_1\otimes \mu^{d+1}(q,v_1,
\textbf{a}_{d-1})  \\
\mu^d(y_4, \textbf{a}_{d-1}) + \mu^{d+1}(y_2,v_2, \textbf{a}_{d-1}) +
\mu^{d+1}(y_3, v_3, \textbf{a}_{d-1}) + \mu^{d+2}(y_1, q, v_1,
\textbf{a}_{d-1}) \end{array}
\right)}
\end{multline*}
for $d\geq2$.

Without loss of generality we may assume that $\VV(V)$ is two-dimensional with basis $\{e_V,h\}$ so that we may write $\YY(V) \otimes \VV(V)[2]$ as a direct sum $e[2]\YY(V) \oplus h\YY(V)$ and denote by $\pi_h$ the projection onto the
second summand (without any correcting sign factor).

For all $\YY$, we now define maps $\alpha_{\YY} \colon T^2_V\YY \to \Phi_V\YY$ by
\[
\alpha_{\YY}^1 \left( \begin{array}{cccc}
y_1 \otimes q \otimes v_1\\
y_2 \otimes v_2\\
y_3 \otimes v_3\\
y_4 \end{array}
\right)
=
\left( \begin{array}{ccc}
(-1)^{|v_1|}\pi_h(y_1 \otimes q) \otimes v_1\\
(-1)^{|y_2|+|v_2|}y_2 \otimes v_2 + (-1)^{|y_3|+|v_3|}y_3 \otimes v_3\\
(-1)^{|y_4|-1}y_4 \end{array}
\right),
\]
and, given $t \in \hom_{\QQ}(\YY,\zz)$, we now have the diagram
\begin{equation*}
\xymatrix{
{T^2_V\YY} \ar[r]^{\tilde{\tilde{t}}} \ar[d]^{\alpha_{\YY}} & {T^2_V\zz} \ar[d]^{\alpha_{\zz}} \\
{\Phi_V\YY} \ar[r]^{\hat{t}}  & {\Phi_V\zz,}
}
\end{equation*}
and the following are easily checked:
\begin{itemize}
\item $\mu_{\QQ}^1\left(\alpha_{\YY}\right)=0$ for all $\YY$;
\item By a similar argument to the proof of Proposition \ref{thm:shift}, $\alpha_{\YY}$ is a quasi-isomorphism for all $\YY$;
\item $(-1)^{|\tilde{\tilde{t}}|}\mu_{\QQ}^2\left(\alpha_{\zz},\tilde{\tilde{t}}\right) = (-1)^{|\alpha_{\YY}|}\mu_{\QQ}^2\left(\hat{t},\alpha_{\YY}\right)$.
\end{itemize}
This suffices to prove that there is a natural isomorphism between the two functors in $D\AA$.
\end{proof}

\begin{cor}
In light of \eqref{eqn:symp}, Conjecture \ref{thm:conj} holds in the case of a $\CC\pp^1$-object.
\end{cor}

\subsection{\texorpdfstring{Non-surjectivity of \eqref{eqn:symptofuk}}{Non-surjectivity of (1.1)}}

Suppose we have a symplectic manifold $\mo$ and a Lagrangian $V \subset M$ which satisfies the classical ring isomorphism $HF^*(V,V) \cong H^*(V)$.  Then if $V$ has the same cohomology ring as $\CC\pp^n$ we can form the projective twist $\Phi_V$ of $D\FF(M)$ even if $V$ is not itself diffeomorphic to $\CC\pp^n$.  However, in this case we would not expect to find a geometric representative of $\Phi_V$ as we do not expect to find a metric on $V$ with periodic geodesic flow.  We shall prove that there are indeed situations as above where no such geometric twist exists.  The argument in this section is very similar to that in \cite[Proposition 2.17]{abouzaidsmithplumbings} and we refer the reader there for a more precise discussion of the technical issues underpinning the definition of the Fukaya category in this situation. 

Take some manifold $V$ such that $H^*(V) \cong k[h]/h^{n+1}$ as a ring but such that $\pi_1(V)$ is nontrivial (for example we could take the connect sum of $\CC\pp^n$ and some homology sphere $\Sigma^{2n}$).  Consider the disc cotangent bundle $D^*V$ and add a Weinstein handle \cite{weinsteinhandle} to cap off the Legendrian $S^{2n-1}$ bounding some cotangent fibre.  The result is an exact symplectic manifold $M=D^*V \# D^*S^{2n} $, which contains Lagrangians $Y \cong S^{2n}$ and $V$, and results of \cite{abouzaidsmithplumbings} say that (for some suitable definition of the Fukaya category) $\FF(M)$ is generated (not merely split-generated) by these two Lagrangians.  Moreover, here we have the identity $HF^*(V,V) \cong H^*(V)$.

\begin{prop}
In this situation there is no geometric representative $\phi_V$ of $\Phi_V$.
\end{prop}

We first fix the coefficient field $\KK$ we shall use to define our Fukaya category: let $\iota \colon \widetilde{V} \to V$ denote the universal cover and fix some $\KK$ such that $char(\KK)$ divides the index of $\iota$ (so that $char(\KK)$ is arbitrary when the index is infinite). Now suppose that such a geometric morphism $\phi_V$ exists.  Then there will be a Lagrangian submanifold $L=\phi_V(Y)$ which is represented by the twisted complex
\begin{equation}
\label{eqn:thing}
\xymatrixcolsep{.2in}
\begin{xymatrix}
{V  \ar[r]^{h} &V[1]  \ar[r]^{x} & Y,}
\end{xymatrix}
\end{equation}
where the arrows denote the terms in the differential as in \eqref{eqn:twistedcxphi} (if necessary we pass to a quasi-equivalent, strictly unital $\widetilde{\FF}(M)$ so that we may work with twisted complexes as in Section \ref{sec:equivalence}).  Here we observe that $HF^*(V,Y) = \KK$ generated by their one point of intersection $x$.  The objects of $\FF(M)$ are all closed Lagrangians, but $\FF(M)$ embeds as a full category of some \emph{wrapped Fukaya category} $\WW(M)$, which includes nonclosed Lagrangians such as cotangent fibres.  Let $\pi \colon \widetilde{M} \to M$ be the cover induced by $\iota \colon \widetilde{V} \to V$.  Results of \cite[Section 6]{abouzaidmaslov} now say that there exists a pullback Fukaya category $\mathcal{W}(\widetilde{M};\pi)$ with the following properties:

\begin{thm}
There is a wrapped Fukaya category $\mathcal{W}(\widetilde{M};\pi)$ which comes with a pullback
functor
\[
\pi^* \colon \mathcal{W}(M) \to \mathcal{W}(\widetilde{M} ; \pi)
\]
which acts on objects $L$ of $\mathcal{W}(M)$ by taking the total inverse image $\pi^{-1}(L) \subset \widetilde{M}$ and such
that the map on morphisms
\[
HF^*(L,L) \to HF^*(\pi^{-1}(L),\pi^{-1}(L))
\]
agrees with the classical pullback on cohomology whenever $L \subset M$ is closed. Moreover, deck
transformations of $\pi$ act by autoequivalences of $\mathcal{W}(\widetilde{M} ;\pi)$.
\end{thm}

So when we pullback the twisted complex \eqref{eqn:thing} under $\pi$, we get a new twisted complex in $\WW(\widetilde{M};\pi)$:
\[
\xymatrixcolsep{.2in}
\begin{xymatrix}
{\widetilde{V} \ar[r]^{0} &\widetilde{V}[1] \ar[r] & \pi^{-1}(Y),}
\end{xymatrix}
\]
where the first differential is zero by our choice of $\KK$.  This means that, up to shifts, we get the splitting
\begin{equation}
\label{eqn:splitting}
\pi^{-1}(L) \cong \widetilde{V} \oplus \left(\widetilde{V}[1] \to \pi^{-1}(Y) \right).
\end{equation}
Also, $\pi^{-1}(L) = \coprod_{\alpha} \widetilde{L}_{\alpha}$ where all the components are related in $\mathcal{W}(\widetilde{M} ;\pi)$ by deck transformations of $\pi$.  By looking at the rank of $HW^0(\widetilde{V},\widetilde{V})=HF^0(\widetilde{V},\widetilde{V})=\KK$ we see that $\widetilde{V}$ is an indecomposable object of the category, as is each $\widetilde{L}_{\alpha}$.

We now work in $D^{\pi}\WW(\widetilde{M};\pi)$, the idempotent completion of $D\WW(\widetilde{M};\pi)$ \cite[Chapter 4]{seidelbible}, where we can appeal to the following lemma.

\begin{lem}
If $X=\oplus X_i$ is a direct sum of indecomposable objects in $D^{\pi}\WW(\widetilde{M};\pi)$ and $Y$ is a indecomposable summand of $X$, then $Y$ must be isomorphic to one of the $X_i$.
\end{lem}

\begin{proof}
By considering inclusion and projection morphisms, we see that the composition $X \to Y \to X$ is idempotent.  This splits as a direct sum of idempotents $X_i \to Y \to X_i$.  When one of these is nonzero it means that, either the composition is the identity or that, having taken idempotent completion, $X_i$ admits a nontrivial summand.  In the first instance, $Y \to X_i \to Y$ is then idempotent, so again the composition is either the identity or $Y$ admits a nontrivial decomposition.  As $X_i$ and $Y$ are assumed indecomposable, we conclude that $X_i$ and $Y$ must be isomorphic. 
\end{proof}

Therefore, in order to show that the twisted complex in the right-hand side of \eqref{eqn:splitting} cannot arise as the pullback of a geometric Lagrangian and that therefore $\phi_V$ cannot exist, it suffices to prove

\begin{lem}
$\widetilde{V}[1] \to \pi^{-1}(Y)$ is not quasi-isomorphic in $\WW(\widetilde{M};\pi)$ to a direct sum of objects obtained from $\widetilde{V}$ by deck transformations.
\end{lem}

\begin{proof}
Pick a cotangent fibre to one of the components of $\pi^{-1}(Y)$ and consider its Floer cohomology with these two twisted complexes.  In the case of $\widetilde{V}$ the rank will be zero; in the case of $\widetilde{V}[1] \to \pi^{-1}(Y)$ the rank will be 1.
\end{proof}

\begin{rmk}
\label{rmk:relspin}
This argument requires that we may freely choose our coeffiecient field for $\FF(M)$.  To do this one usually restricts attention to \emph{spin} Lagrangians so that we can orient the moduli spaces of holomorphic curves used to define our $\ainf$-maps.  However, following \cite{fooo1}, it is enough that our Lagrangians be \emph{relatively spin}, meaning that there is some class $st \in H^*(M,\ZZ/2)$ such that $st|_L = w_2(L)$, which clearly holds here.  Therefore the above argument will still work in the case where $n$ is even.
\end{rmk}


\begin{thebibliography}{10}

\bibitem{abouzaidmaslov}
M.~Abouzaid.
\newblock Nearby {L}agrangians with vanishing {M}aslov class are homotopy
  equivalent.
\newblock Preprint arXiv:1005.0358, 2010.

\bibitem{abouzaidsmithplumbings}
M.~Abouzaid and I.~Smith.
\newblock Exact {L}agrangians in plumbings.
\newblock Preprint arXiv:1107.0129, 2011.

\bibitem{besse}
A.~L. Besse.
\newblock {\em Manifolds all of whose geodesics are closed}, volume~93 of {\em
  Ergebnisse der Mathematik und ihrer Grenzgebiete}.
\newblock Springer-Verlag, Berlin, 1978.

\bibitem{bott}
R.~Bott.
\newblock On manifolds all of whose geodesics are closed.
\newblock {\em Ann. of Math. (2)}, 60:375--382, 1954.

\bibitem{bridgeland}
T.~Bridgeland.
\newblock Equivalences of triangulated categories and {F}ourier-{M}ukai
  transforms.
\newblock {\em Bull. London Math. Soc.}, 31(1):25--34, 1999.

\bibitem{fooo1}
K.~Fukaya, Y.-G. Oh, H.~Ohta, and K.~Ono.
\newblock {\em Lagrangian intersection {F}loer theory: anomaly and obstruction.
  {P}arts {I} and {II}}, volume~46 of {\em AMS/IP Studies in Advanced
  Mathematics}.
\newblock American Mathematical Society, Providence, RI, 2009.

\bibitem{grant}
J.~Grant.
\newblock Derived autoequivalences from periodic algebras.
\newblock Preprint arXiv:1106.2733, 2011.

\bibitem{huybrechtsthomas}
D.~Huybrechts and R.~Thomas.
\newblock {$\mathbb{P}$}-objects and autoequivalences of derived categories.
\newblock {\em Math. Res. Lett.}, 13(1):87--98, 2006.

\bibitem{seidelgradings}
P.~Seidel.
\newblock Graded {L}agrangian submanifolds.
\newblock {\em Bull. Soc. Math. France}, 128(1):103--149, 2000.

\bibitem{seidelles}
P.~Seidel.
\newblock A long exact sequence for symplectic {F}loer cohomology.
\newblock {\em Topology}, 42(5):1003--1063, 2003.

\bibitem{seidelbible}
P.~Seidel.
\newblock {\em Fukaya categories and {P}icard-{L}efschetz theory}.
\newblock Zurich Lectures in Advanced Mathematics. European Mathematical
  Society (EMS), Z\"urich, 2008.

\bibitem{weinsteinhandle}
A.~Weinstein.
\newblock Contact surgery and symplectic handlebodies.
\newblock {\em Hokkaido Math. J.}, 20(2):241--251, 1991.
   
\end{thebibliography}
\end{document}